\documentclass[11pt, twoside]{article}
\usepackage{amstext,amsfonts,amsthm,amsmath,amssymb}
\usepackage[top=30mm, left=25mm, right=25mm, bottom=27mm]{geometry}
\usepackage{inputenc}
\usepackage{fancyhdr}
\usepackage{tikz}
\usetikzlibrary{positioning}
\usepackage{titling}
\usepackage{url}
\usepackage{hyperref}
\hypersetup{hidelinks,
 pdfstartview=FitH,
 colorlinks=true,
 linkcolor=blue,
 citecolor=blue,
 urlcolor=blue}
\usepackage[T1]{fontenc}
\usepackage{currvita}
\usepackage{xcolor}
\usepackage[capitalize]{cleveref}
\usepackage{subcaption}
\predate{}
\postdate{}
\usepackage{lipsum}
\newtheorem{definition}{Definition}
\newtheorem{theorem}{Theorem}
\newtheorem{corollary}[theorem]{Corollary}

\newtheorem{lemma}[theorem]{Lemma}
\newtheorem{proposition}[theorem]{Proposition}
\newtheorem{claim}[theorem]{Claim}

\newtheorem{conjecture}[theorem]{Conjecture}
\DeclareMathOperator{\sat}{sat*}
\begin{document}
\newcommand{\Addresses}{{
\bigskip
\footnotesize
\medskip

Paul~Bastide, \textsc{LaBRI, Université de Bordeaux, Bordeaux, France}\par\nopagebreak\textit{Email address: }\texttt{paul.bastide@ens-rennes.fr}

\medskip

Carla~Groenland, \textsc{Institute of Applied Mathematics, Technische Universiteit Delft (TU Delft), 2628 CD Delft, Netherlands. }\par\nopagebreak\textit{Email address: }\texttt{c.e.groenland@tudelft.nl}

\medskip

Maria-Romina~Ivan, \textsc{Magdalene College, University of Cambridge, Cambridge, CB3 0AG, UK and Department of Pure Mathematics and Mathematical Statistics, Centre for Mathematical Sciences, Wilberforce Road, Cambridge, CB3 0WB, UK.}\par\nopagebreak\textit{Email address: }\texttt{mri25@dpmms.cam.ac.uk}

\medskip

Tom~Johnston, \textsc{School of Mathematics, University of Bristol, Bristol,      BS8 1UG, UK and Heilbronn Institute for Mathematical Research, Bristol, UK.}\par\nopagebreak\textit{Email address: }\texttt{tom.johnston@bristol.ac.uk}}}
\pagestyle{fancy}
\fancyhf{}
\fancyhead [LE, RO] {\thepage}
\fancyhead [CE] {PAUL BASTIDE, CARLA GROENLAND, MARIA-ROMINA IVAN AND TOM JOHNSTON}
\fancyhead [CO] {A POLYNOMIAL UPPER BOUND FOR POSET SATURATION}
\renewcommand{\headrulewidth}{0pt}
\renewcommand{\l}{\rule{6em}{1pt}\ }
\title{\Large\textbf{A POLYNOMIAL UPPER BOUND FOR POSET SATURATION}}
\author{PAUL BASTIDE, CARLA GROENLAND,\\ MARIA-ROMINA IVAN AND TOM JOHNSTON}
\date{}
\maketitle
\newcommand{\floor}[1]{\left\lfloor #1 \right\rfloor}
\newcommand{\ceil}[1]{\left\lceil #1 \right\rceil}
\newcommand{\A}{\mathcal{A}}
\newcommand{\B}{\mathcal{B}}
\newcommand{\F}{\mathcal{F}}
\newcommand{\C}{\mathcal{C}}
\newcommand{\I}{\mathcal{I}}
\newcommand{\Q}{\mathcal{Q}}
\newcommand{\T}{\mathcal{T}}
\renewcommand{\S}{\mathcal{S}}
\newcommand{\R}{\mathcal{R}}
\newcommand{\D}{\mathcal{D}}
\newcommand{\X}{\mathcal{X}}
\renewcommand{\P}{\mathcal{P}}

\newcommand{\tom}[1]{\textcolor{red}{#1}}
\newcommand{\carla}[1]{\textcolor{blue}{#1}}
\newcommand{\maria}[1]{\textcolor{purple}{#1}}
\newcommand{\paul}[1]{\textcolor{orange}{[#1]}}

\begin{abstract}
Given a finite poset $\mathcal P$, we say that a family $\mathcal F$ of subsets of $[n]$ is $\mathcal P$-saturated if $\mathcal F$ does not contain an induced copy of $\mathcal P$, but adding any other set to $\mathcal F$ creates an induced copy of $\mathcal P$. The induced saturation number of $\mathcal P$, denoted by $\text{sat}^*(n,\mathcal P)$, is the size of the smallest $\mathcal P$-saturated family with ground set $[n]$. In this paper we prove that the saturation number for any given poset grows at worst polynomially. More precisely, we show that $\text{sat}^*(n, \mathcal P)=O(n^c)$, where $c\leq|\mathcal{P}|^2/4+1$ is a constant depending on $\mathcal P$ only. We obtain this result by bounding the VC-dimension of our family.
\end{abstract}

\section{Introduction}
\par We say that a poset $(\mathcal Q,\preceq)$ contains an \textit{induced copy} of a poset $(\mathcal P,\preceq')$ if there exists an injective order-preserving function $f:\mathcal P\rightarrow\mathcal Q$ such that $(f(\mathcal P),\preceq)$ is isomorphic to $(\mathcal P,\preceq')$. We denote by $2^{[n]}$ the power set of $[n]=\lbrace 1,2,\dots,n\rbrace$. We define the \textit{$n$-hypercube}, denoted by $Q_n$ to be the poset formed by equipping $2^{[n]}$ with the partial order induced by inclusion.
\par If $\mathcal P$ is a finite poset and $\mathcal F$ is a family of subsets of $[n]$, we say that $\mathcal F$ is $\mathcal P$-\textit{saturated} if $\mathcal F$ does not contain an induced copy of $\mathcal P$, and for any $S\notin\mathcal F$, the family $\mathcal F\cup S$ contains an induced copy of $\mathcal P$. The smallest size of a $\mathcal P$-saturated family of subsets of $[n]$ is called the \textit{induced saturated number}, and denoted by $\text{sat}^*(n,\mathcal P)$.
\par It has been shown that the growth of $\text{sat}^*(n,\mathcal P)$ has a dichotomy. Keszegh, Lemons, Martin, P{\'a}lv{\"o}lgyi and Patk{\'o}s \cite{keszegh2021induced} proved that for any poset the induced saturated number is either bounded or at least $\log_2(n)$.  They also conjectured that in fact sat$^*(n,\mathcal P)$ is either bounded, or at least $n+1$. Recently, Freschi, Piga, Sharifzadeh and Treglown \cite{freschi2022induced} improved this result by replacing $\log_2 (n)$ with $2\sqrt{n-2}$. There is no known poset $\mathcal P$ for which $\text{sat}^*(n,\mathcal P)=\omega(n)$, and it is in fact believed that for any poset, the saturation number is either constant or grows linearly.
\par Whilst, as summarised above, some general lower bounds have been established, no non-trivial general upper bounds have yet been found. Given a general poset $\mathcal P$, what can we say about upper bounds on its saturation number? How fast can it grow? Is it possible to have an intricate partial relation that forces the saturation number to grow faster than any polynomial? The aim of this paper is to show that the answer is no: the saturation numbers have at worst polynomial growth. Our main result is the following.
\begin{theorem}
Let $\mathcal P$ be a finite poset, and let $|\mathcal P|$ denote the size of the poset. Then $\textnormal{sat}^*(n,\mathcal P)=O(n^c)$, where $c\leq|\mathcal P|^2/4+1$ is a constant depending on $\mathcal P$ only.
\label{thm:main}
\end{theorem}
\par Induced and non-induced poset saturation numbers are a growing area of study in combinatorics. Saturation for posets was introduced by Gerbner, Keszegh, Lemons, Palmer, P{\'a}lv{\"o}lgyi and Patk{\'o}s~\cite{gerbner2013saturating}, although this was not for \textit{induced} saturation. Induced poset saturation was first introduced in 2017 by Ferrara, Kay, Kramer, Martin, Reiniger, Smith and Sullivan \cite{ferrara2017saturation}. We briefly summarise some of the recent developments below, and also refer the reader to the textbook of Gerbner and Patk{\'o}s \cite{gerbner2019extremal} for a nice introduction to the area.

\begin{figure}
\centering
\begin{subfigure}{0.4\textwidth}
\centering
\begin{tikzpicture}
        \node (A) at (-7,0) {\textbullet};
        \node (B) at (-7,2) {\textbullet};
        \node (C) at (-5,0) {\textbullet};
        \node (D) at (-5,2) {\textbullet};
        \draw (A)--(B)--(C)--(D)--(A); 
\end{tikzpicture}
\caption{}
\label{fig:butterfly}
\end{subfigure}%
\begin{subfigure}{0.4\textwidth}
\centering
\begin{tikzpicture}
    \node (top) at (-1,2) {$\bullet$}; 
    \node (left) at (-2,1) {$\bullet$};
    \node (right) at (0,1) {$\bullet$};
    \node (bottom) at (-1,0) {$\bullet$};
    \draw (top)--(right)--(bottom)--(left)--(top);
\end{tikzpicture}
\caption{}
\label{fig:diamond}
\end{subfigure}
\caption{The butterfly poset $\mathcal{B}$ and the diamond poset $\mathcal{D}_2$.}
\end{figure}
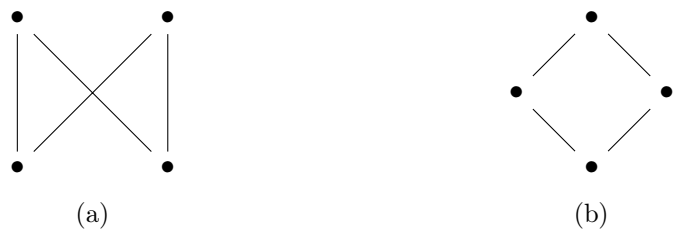
\par Determining the saturation number, even for small posets, has proven to be a difficult question. The \emph{exact} saturation number is known for only a precious few posets such as the $X$ and $Y$ posets~\cite{freschi2022induced}, chains with at most 6 sets \cite{morrison2014saturated}, and the fork \cite{ferrara2017saturation}. The only class of large posets for which exact saturation numbers are known are the $k$-antichains, denoted by $\mathcal{A}_k$.  It is easy to see that a collection of $k-1$ full chains (chains of order $n+1$) that intersect only at $\emptyset$ and $[n]$ form a $k$-antichain saturated family. Thus, for $n$ large enough, we certainly have $\text{sat}^*(n, \mathcal A_{k})\leq (k-1)(n-1)+2$. In the other direction, Martin, Smith and Walker \cite{martin2020improved} showed that for $k\geq4$ and $n$ large enough $\text{sat}^*(n,\mathcal A_{k})\geq \left(1-\frac{1}{\log_2(k-1)}\right)\frac{(k-1)n}{\log_2(k-1)}$. Recently, Bastide, Groenland, Jacob and Johnston \cite{bastide2022exact} showed that $\text{sat}^*(n,\mathcal A_k)=(k-1)n-\Theta(k\log k)$, and gave the exact value for $n$ sufficiently large compared to $k$. 
\par Other posets that have received special attention are the butterfly (Figure \ref{fig:butterfly}), which we denote by $\mathcal B$, and the diamond (Figure \ref{fig:diamond}), which we denote by $\mathcal D_2$. The butterfly poset is at least known to be linear, but the upper and lower bounds differ by a constant factor. Indeed, the best known lower bound is $\sat(n, \mathcal{B}) \geq n +1$ as shown by Ivan in \cite{ivan2020saturation}, while the best upper bound is currently $\sat(n, \mathcal{B}) \leq 6n - 10$, as shown by Keszegh, Lemons, Martin, P{\'a}lv{\"o}lgyi and Patk{\'o}s in \cite{keszegh2021induced}. Even less is known about the diamond. Martin, Smith and Walker \cite{martin2020improved} proved that $\sqrt n\leq\text{sat}^*(n, \mathcal D_2)\leq n+1$. The lower bound was later improved by Ivan \cite{ivan2022minimal} and now stands at $\text{sat}^*(n, \mathcal D_2) \geq (2\sqrt2-o(1))\sqrt n$. Despite the simple structure of the diamond, whether its saturation number is linear is still unknown.

\medskip
\medskip

\par The proof of Theorem \ref{thm:main} uses the following two new key notions, `cube-height' and `cube-width'. For a poset $\mathcal{P}$, the `cube-height' is the least $k$ such that, for some $n$, we can embed $\mathcal P$ into the first $k+1$ layers of $Q_n$, while the `cube-width' is the smallest $n$ that makes such a `small height' embedding possible. We give formal definitions and bounds on these two notions in Section 2. The cube-height and cube-width are designed to build a $\mathcal P$-saturated family with bounded VC-dimension. This is done in Section 3, where we prove Theorem \ref{thm:main}.
\par Our construction could be viewed as the result of a greedy algorithm where the sets are ordered according to size (and then arbitrarily within the layers), and an element is added to the family as long as it does not create a copy of $\P$ in the family. Greedy algorithms have been used before for studying poset saturation; most notably, a greedy colex algorithm was used to show a linear upper bound for the butterfly \cite{keszegh2021induced}. 
Our result shows that `layer-by-layer' greedy algorithms result in a saturated family of size at most $n^{|\P|^2/4 + 1}$, and we note that such an algorithm has a near-linear time complexity of $O_\P(|Q_n| (\log_2|Q_n|)^{|\P|^3})=O_\P(2^n n^{|\P|^3})$. This follows from the fact that for any family $\F$, it can be decided if it is $\P$-free in $O_\P(|\F|^{|\P|})$ time.
\par We end the introduction by reminding the reader about the VC-dimension of a family of sets. We say that a family $\F$ of subsets of $[n]$ \emph{shatters} a set $S \subseteq [n]$ if, for all $F\subseteq S$, there exists $A\in\F$ such that $A\cap S=F$. In other words, $\{A\cap S:A\in \F\}$ is the power set of $S$. The \emph{VC-dimension} of $\F$ is the largest cardinality of a set shattered by $\F$. The size of a family $\mathcal F$ with bounded VC-dimension grows at worst polynomially, as shown by the following well-known result.
\begin{lemma}[Sauer-Shelah lemma \cite{sauer1972density,shelah1972combinatorial}] If $\F\subseteq 2^{[n]}$ has VC-dimension $d$, then $|\F|\leq\displaystyle\sum_{i=0}^d \binom{n}{i}$.
\label{lem:ss}
\end{lemma}
\section{Cube-height and cube-width}
\par In this section we discuss how to `fit' a given poset $\mathcal P$ into a hypercube. We do this with the help of cube-height and cube-width, the two new quantities mentioned above, which we bound in terms of $|\mathcal{P}|$.  Given two integers $h\leq w$, we denote by $\binom{[w]}{\leq h}$ the induced subposet of the hypercube $Q_w$ consisting  of all the sets of size at most $h$, i.e the poset $Q_w$ restricted to the first $h + 1$ layers, $0,1,\dots,h$. 
\begin{definition} 
For a poset $\mathcal P$, we define the \textit{cube-height} $h^*(\mathcal P)$ to be the minimum $h^* \in \mathbb{N}$ for which there exists $n \in \mathbb{N}$ such that $\binom{[n]}{\leq h^*}$ contains an induced copy of $\mathcal{P}$.
\end{definition}
\begin{definition} For a poset $\mathcal P$, we define the \textit{cube-width} $w^*(\mathcal P)$ to be the minimum $w^* \in \mathbb{N}$ such that there exists an induced copy of $\mathcal P$ in $\binom{[w^*]}{\leq h^*(\mathcal P)}$.\end{definition}
\par We stress that the two notions defined above are different from the usual height and width of $\mathcal P$, that is, from the size of the biggest chain and antichain, respectively. It is easy to see that the height of $\mathcal P$ is always at most $h^*(\mathcal P)+1$, and that equality can happen (e.g. for a chain), but that is not always the case. Indeed, if $\mathcal P$ is the butterfly poset (Figure \ref{fig:butterfly}), then the height of $\mathcal P$ is 2 and its cube-height is 3: in any hypercube, the first 3 layers are butterfly-free.
\par Similarly, the width and the cube-width can be very different. For example, if $\mathcal P$ is a chain of size $k$, then its width is 1, but its cube-width is $k-1$. Cube-width is not even a monotone property. For example, the antichain of size $\binom{k}{k/2}$ has cube-height 1 and cube-width $\binom{k}{k/2}$, but adding a chain of length $k/2$ which is less than all elements of the antichain gives a poset with cube-height $k/2$ and cube-width $k$.
\par It is important to remark that the cube-width is \textit{not} the minimal $n$ for which the poset can be embedded in $Q_n$. Indeed, the cube-width of an antichain of size 20 is 20, but $Q_6$ contains an antichain of size 20, namely the middle layer.
\medskip

\par We now bound the cube-height and cube-width in terms of the size of the poset.
\begin{lemma}
For any poset $\mathcal P$, we have that $h^*(\mathcal P)\leq |\mathcal P|-1$.
\label{lem:height}
\end{lemma}
We remark that the inequality in this lemma is tight, since a chain on $k$ elements has cube-height $k-1$.
\begin{proof}[Proof of Lemma \ref{lem:height}] 
\par We prove that any poset $\mathcal{P}$ on $k$ elements embeds in $\binom{[n]}{\leq k-1}$ for all $n\geq k$ by induction on $k$.
\par The base case $k=1$ is trivially true since the cube-height of a poset with 1 element is $0$. Let $k\geq 2$ and assume the claim is true for all posets of size less than $k$.
\par Let $\mathcal P=(\{p_1,\dots,p_k\},\preceq)$, and suppose that $n\geq k$. We show that $\mathcal P$ appears as an induced poset in $\binom{[n]}{\leq k-1}$. 
\par Suppose first that $\mathcal{P}$ has a unique maximal element. After renumbering the elements as necessary, we may assume that $p_k$ is the unique maximal element of $\mathcal P$. In this case, using the induction hypothesis, we find sets $A_1,\dots, A_{k-1}\in \binom{[k-1]}{\leq k-2}$ such that they induce a copy of the poset $\mathcal{P}\setminus \{p_k\}$. Now let $A_k=[k-1]$ and observe that $A_1,\dots,A_k$ induce a copy of $\mathcal{P}$ in $\binom{[n]}{\leq k-1}$. 
Indeed, $A_i \subsetneq A_k$ for all $i\leq k-1$ since $A_k$ has size $k-1$, while $|A_i|\leq k-2$ for all $i\leq k$.
\par Suppose now that $\mathcal{P}$ does not contain a unique maximal element. We construct the sets $A_1,\dots,A_k$ as follows: for any $i,j \in [k]$, $i\in A_j$ if and only if $p_i\preceq p_j$. We observe that all constructed sets are subsets of $[k]\subseteq [n]$ and the size of each $A_i$ is the number elements less than or equal to $p_i$ (including $p_i$), which is at most $k-1$ since $\mathcal{P}$ has no unique maximal element. It remains to argue that $\{A_1,\dots, A_k\}$ induces a copy of $\mathcal P$ in $Q_n$.
\par If $p_i\not \preceq p_j$, then $A_i\not \subseteq A_j$ since $i\in A_i\setminus A_j$. On the other hand, if $p_i\preceq p_j$, then $\ell \in A_i$ implies $p_\ell\preceq p_i$, which implies $p_\ell\preceq p_j$ by transitivity. Therefore, $\ell \in A_j$, which shows that $A_i\subseteq A_j$, as required.
\end{proof}
\par Note that the proof above gives a simple algorithm for constructing an embedding.
\par In the lemma above, we embedded $\mathcal P$ into $Q_{|\mathcal P|}$. This cannot be improved in general, as seen in the following example.
Let $\mathcal P_t$ be the poset consisting of $t$ antichains $\mathcal A_1,\dots,\mathcal A_t$ of size $2$, where we further impose that any element of $\mathcal A_i$ is less than any element of $\mathcal A_j$ for all $i<j$. 
Now, $2t=|\mathcal P_t|$ and by induction on $t$ it follows that $\mathcal P_t$ does not embed into $Q_{2t-1}$. Indeed, since everything below one of these antichains is a subset of both of its elements, each $\mathcal A_i$ must use at least two new elements of the ground set. 
Thus, this poset can be embedded in $Q_{2t}$, but not in $Q_{2t-1}$.

\medskip
\par We say that a collection of sets $A_1,\dots,A_k\subseteq [n]$ forms an \emph{optimal cube-height embedding of $\mathcal P$} if they are pairwise distinct and induce a copy of $\mathcal{P}$ in $\binom{[n]}{\leq h^*(\mathcal P)}$.
By the definition of $h^*(\mathcal P)$ such an embedding exists, and its ground set is $A_1\cup \dots \cup A_k$, which has size at most $h^*(\mathcal{P})k$, thus we immediately get the following corollary.
\begin{corollary} 
\label{lem:width_easy}
For any poset $\mathcal{P}$, $w^*(\mathcal{P})\leq h^*(\mathcal{P})|\mathcal{P}|\leq |\mathcal{P}|^2$.
\end{corollary}
\par This corollary can immediately be strengthened by noting that we only need to take the union over the maximal elements in the embedding, so $w^*(\mathcal{P})$ is bounded by $h^*(\mathcal{P})$ times the number of maximal elements. Since the number of maximal elements is bounded by the size of the largest antichain in $\mathcal{P}$, denoted by $w(\mathcal P)$, this gives the following bound 
$$w^*(\mathcal{P})\leq h^*(\mathcal{P})w(\mathcal{P}).$$
\par Whilst Corollary \ref{lem:width_easy} is enough for us to prove that $\sat(n, \mathcal{P}) = O(n^{|\mathcal{P}|^2 - 1})$, proving Theorem~\ref{thm:main} requires a stronger bound on $w^*(\mathcal{P})$, which is given by the following lemma.

\begin{lemma}
\label{lem:width}
For any poset $\mathcal{P}$, we have that $w^*(\mathcal{P}) \leq |\mathcal{P}|^2/4 + 2$.
\end{lemma}
\par In order to prove Lemma \ref{lem:width} we will make use of Lemma 2.2 from \cite{freschi2022induced}, which we state below for completeness.
\begin{lemma}[\cite{freschi2022induced}, Lemma 2.2]
\label{lemma:treglown}
Let $\mathcal{F} \subseteq 2^{[n]}$ be such that for every $i \in [n]$ there exist two elements $A, B \in \mathcal{F}$ such that $A \setminus B = \{i\}$. Then $|\mathcal{F}| \geq 2\sqrt{n - 2}$.
\end{lemma}
\par The next proposition tells us that any optimal cube-height embedding has the property stated in Lemma \ref{lemma:treglown}. We remark that this property in itself may be of independent interest, as explained in the final section.
\begin{proposition}
\label{prop:property}
Let $\mathcal P$ be a poset and let $A_1,\dots,A_k\in \binom{[w^*(\mathcal{P})]}{\leq h^*(\mathcal{P})}$ be distinct sets that induce a copy of $\mathcal P$. Then for all $a\in [w^*(\P)]$, there exist $i,j\in [k]$ such that $A_i\setminus A_j =\{a\}$.
\end{proposition}
\begin{proof}
Suppose there exists $a \in [w^*(\mathcal{P})]$ such that there does not exist $i,j \in [k]$ with $A_i\setminus A_j =\{a\}$. By relabelling as necessary, we may assume that $a = w^*(\mathcal{P})$, which we denote by $w^*$ for clarity. We now replace $A_i$ by $A_i \setminus \{w^*\}$ for all $i\leq k$. This new family lives in $\binom{[w^*-1]}{\leq h^*(\mathcal P)}$ and we claim it still forms a copy of $\mathcal{P}$. First, notice that we do not decrease the size of the family: for that to happen there would have to be distinct $A_i, A_j$ such that $A_j = A_i \cup \{w^*\}$, but that  would immediately imply $A_j \setminus A_i = \{w^*\}$, a contradiction. We are left to show that comparability and incomparability relations are preserved. Let $A_i, A_j$ be such that $A_i \subseteq A_j$. Then $A_i\setminus \{w^*\} \subseteq A_j \setminus \{w^*\}$, as required. Finally, let $A_i$ and $A_j$ be incomparable, and assume that $A_i\setminus \{w^*\} \subseteq A_j\setminus \{w^*\}$. This implies that $w^*\in A_i$ and $w^*\notin A_j$, and consequently $A_i\setminus A_j=\{w^*\}$, a contradiction. Therefore, $A_i\setminus\{w^*\}$ and $A_j\setminus\{w^*\}$ are incomparable, and we have indeed shown that the new family forms an induced copy of $\mathcal P$. However, this new family lives in $\binom{[w^*-1]}{\leq h^*(\mathcal P)}$, contradicting the definition of $w^*$. 
 \end{proof}
\par We are now ready to prove the stronger upper bound on $w^*(\mathcal P)$.
\begin{proof}[Proof of Lemma \ref{lem:width}] Suppose $\mathcal F=\{A_1, \dots, A_k\}$ forms an optimal cube-height embedding of $\mathcal{P}$ in $Q_{w^*(\mathcal P)}$. By Proposition \ref{prop:property}, $\mathcal{F}\subseteq 2^{[w^*(\mathcal{P})]}$ is a family of sets such that, for every $a \in [w^*(\mathcal{P})]$, there exist two sets $A_i, A_j \in \mathcal{F}$ with $A_j \setminus A_i = \{a\}$. Lemma \ref{lemma:treglown} then implies that $|\mathcal{P}| = |\mathcal{F}| \geq 2 \sqrt{w^*(\mathcal{P}) - 2}$, and rearranging $w^*({\mathcal P})\leq |\mathcal P|^2/4+2$.\end{proof}
\section{Proof of the main result}
\par In this section we prove our main result, Theorem \ref{thm:main}. Given a poset $\mathcal{P}$ and $n$ large enough, we will construct a $\mathcal P$-saturated family in $Q_n$ of size at most $2n^{w^*(\mathcal{P}) - 1}$ which, combined with the bound on the cube-width from the previous section, achieves the claimed result.
\begin{proof}[Proof of Theorem \ref{thm:main}]
\par Let $h^*=h^*(\mathcal P)$, $w^*=w^*(\mathcal P)$, and assume $n\geq2 w^*$. Let $\mathcal F_0$ be the family consisting of the first $h^*$ layers, or in other words, all the elements of size at most $h^*-1$. By the definition of the cube-height, the family $\mathcal F_0$ does not contain an induced copy of $\mathcal P$. We now extend this family to a $\mathcal P$-saturated family in an arbitrary fashion. Let $\mathcal F$ be this resulting family. The crucial property of this family is the following.
\begin{claim}
\label{claim:vc}
The VC-dimension of $\mathcal{F}$ is less than $w^*$.
\end{claim}
\begin{proof}\par Suppose towards a contradiction that $\mathcal{F}$ shatters a set $S$ of size $w^*$. By definition, this means that $\mathcal L=\{A\cap S: A\in \mathcal F\}$ is the power set of $S$, and it is isomorphic to $Q_{w^*}$. Since $w^*$ is the cube-width of $\mathcal P$, we can find a copy of $\mathcal P$ in $\mathcal L$ such that all sets have size at most $h^*$. For simplicity, we call this copy $\mathcal P$.
\par Let $M_1,\dots, M_s$ be the sets in this copy of $\mathcal P$ that have size exactly $h^*$ -- they are subsets of $S$ by construction. Let $\mathcal P'=\mathcal P\setminus\{M_1,\dots, M_s\}$. Since we have removed all elements of $\mathcal P$ of maximal size, the height of $\mathcal{P}'$ is less than that of $\mathcal{P}$ (i.e. $h^*(\mathcal P') \leq h^*(\mathcal P) - 1$), and $\mathcal{P}'$ is embedded in the first $h^*$ layers. Hence, the subposet $\mathcal P'$ is contained in $\mathcal F_0\subseteq \mathcal F$.
\par Since each $M_i$ is a subset of $S$, we can find $A_i\in\mathcal F$ such that $A_i\cap S=M_i$ for all $i\leq s$. Note that this implies that $|A_i|\geq h^*$ for all $i\leq s$, and so no $A_i$ appears in $\mathcal P'$. We now show that $\mathcal P'\cup\{A_1,\dots,A_s\}$ is an induced copy of $\mathcal P$ in $\mathcal F$, which will yield the desired contradiction.
\par First, if $B\in\mathcal P'$ is incomparable to $M_i$, then $B$ is also incomparable to $A_i$. This is because if $B$ is a subset of $A_i$, then it is also a subset of $A_i\cap S=M_i$, a contradiction. Conversely, if $B\in\mathcal P'$ is a subset of $M_i=A_i\cap S$, then it is a subset of $A_i$, too. We also have that $A_i$ and $A_j$ are incomparable for $i\neq j$ as they are incomparable when restricted to $S$. Finally, $A_i$ can never be a subset of $B\in \mathcal{P}'$, since $A_i\cap S=M_i$ is not a subset of $B\cap S=B$. 
\par We conclude that $\mathcal P'\cup\{A_1,\dots,A_s\}$ is an induced copy of $\mathcal P$ in $\mathcal F$. This gives a contradiction, proving that the VC-dimension of $\mathcal{F}$ is strictly less than $w^*$, as desired.
\end{proof}
\par Combining Lemma \ref{lem:ss} and Claim \ref{claim:vc}, we conclude that, as $n\geq 2w^*$,
$$\text{sat}^*(n,\mathcal P)\leq|\mathcal F|\leq \displaystyle\sum_{i=0}^{w^*-1} \binom{n}{i}\leq w^*\frac{n^{w^*-1}}{(w^*-1)!} \leq 2n^{w^*-1}.$$
\par Here we have used that $\frac{m}{(m-1)!}\leq2$ for all $m \in \mathbb{N}$, and that, since $n\geq 2w^*$, the largest binomial coefficient in the above sum is $\binom{n}{w^*-1}$. Finally, Lemma \ref{lem:width} tells us that $w^* \leq |\mathcal P|^2/4+2$, which proves Theorem \ref{thm:main}.
\end{proof}

\section{Concluding remarks and further work}
\par A first very natural question is: how small can the cube-width be? 
The antichain shows that $w^*(\mathcal{P})$ may be as large as $|\mathcal{P}|$. However, for all the posets we have considered, the cube-width is always at most the size of the poset. We conjecture that this has to be true in general.
\begin{conjecture} For any finite poset $\mathcal P$, $w^*(\mathcal P)\leq |\mathcal P|$.
\label{conj:width}
\end{conjecture}
\par Since we proved that $\sat(n,\mathcal{P}) = O(n^{w^*(\mathcal{P}) - 1})$, Conjecture \ref{conj:width} would imply that $\text{sat}^*(n, \mathcal P)=O(n^{|\mathcal P| - 1})$.
That upper bound seems the natural threshold for our VC-dimension approach and indeed our construction may yield families of such a size (e.g. for the chain).
\par To conclude the paper, we expand on perhaps one of the most surprising phenomenon we observed in our work. We say that a family $\F \subsetneq Q_n$ \textit{separates} $[n]$ if for every $i \in [n]$ there exist two sets $A$ and $B$ in $\mathcal F$ such that $A\setminus B=\{i\}$. Freschi, Piga, Sharifzadeh and Treglown \cite{freschi2022induced} showed that if the saturation number of a poset $\mathcal{P}$ is unbounded, then any induced $\mathcal P$-saturated family separates $[n]$ (and therefore is of size $\Omega(\sqrt{n})$). On the other hand, in Proposition \ref{prop:property}, we proved that every optimal cube-height embedding separates its ground set. We also note that Keszegh, Lemons, Martin, P{\'a}lv{\"o}lgyi and Patk{\'o}s \cite{keszegh2021induced} arrived at their $\log_2 (n)$ lower bound via a weaker `separability' property of $\mathcal P$-saturated families. This allowed them to build a complete graph on $n$ vertices covered by complete bipartite graphs, each of these corresponding to exactly one set in the family. 
Their lower bound then follows since $\log_2(n)$ is the biclique cover number for the complete graph on $n$ vertices.
\par It seems that poset saturation and separability properties are in some sense deeply interlinked. In view of this, we feel that improvements towards Conjecture \ref{conj:width} may yield ideas for improvements on the general $\sqrt{n}$ lower bound, or vice versa.
\bibliographystyle{amsplain}
\bibliography{document}
\Addresses
\end{document}